\documentclass[12pt, twoside]{article}

\usepackage{amsmath,amsthm,amssymb}

\usepackage{times}
\usepackage{url}

\usepackage{enumerate}

\def\titlerunning#1{\gdef\titrun{#1}}

\makeatletter

\def\author#1{\gdef\autrun{\def\and{\unskip, }#1}\gdef\@author{#1}}

\def\address#1{{\def\and{\\\hspace*{18pt}}\renewcommand{\thefootnote}{}%
		
		\footnote {#1}}%
	
	\markboth{\autrun}{\titrun}}

\makeatother

\def\email#1{e-mail: #1}

\def\subjclass#1{{\renewcommand{\thefootnote}{}%
		
\footnote{\emph{Mathematics Subject Classification (2020):} #1}}}

\def\keywords#1{\par\medskip
	
	\noindent\textbf{Keywords.} #1}

\newtheorem{theorem}{\bf Theorem}
\newtheorem{lemma}[theorem]{\bf Lemma}
\newtheorem{cor}[theorem]{\bf Corollary}
\newtheorem{prop}[theorem]{\bf Proposition}

\newtheorem{example}[theorem]{\bf Example}
\newtheorem{remark}[theorem]{\bf Remark}

\begin{document}
	
	\baselineskip=17pt
	
	\author{Tim Alderson \and Simeon Ball}

	\titlerunning{Sets of subspaces with restricted hyperplane intersection numbers}
	
	\title{On sets of subspaces with restricted hyperplane intersection numbers}
	
	\date{}
	
	\maketitle
	
	\address{ 
		Simeon Ball \and Departament de Matem\`atiques, Universitat Polit\`ecnica de Catalunya, Barcelona, Spain. The author acknowledges the support of the Spanish Ministry of Science, Innovation and Universities grant PID2023-147202NB-I00.\\
		\email{simeon.michael.ball@upc.edu}
		\and
		Tim Alderson \and Department of Mathematics and Statistics, University of New Brunswick Saint John, Saint John, NB, Canada. The author acknowledges the support of the Natural Sciences and Engineering Research Council of Canada (NSERC), [funding reference number 2019-04103]\\
		Cette recherche a \'{e}t\'{e} financ\'{e}e par le Conseil de recherches en sciences naturelles et en g\'{e}nie du Canada (CRSNG), [num\'{e}ro de r\'{e}f\'{e}rence 2019-04103].\\
		\email{tim.alderson@unb.ca} \\
		29 March 2026.
	}
	
	
	\begin{abstract}
		Let $\mathcal{X}$ be a set of $(h-1)$-dimensional subspaces of $\mathrm{PG}(kh-1,q)$ with the property that every hyperplane contains at most $t$ elements of $\mathcal{X}$. We prove the upper bound $|\mathcal{X}| \leq (t-k+2)q^h + t$, and characterise the structure of $\mathcal{X}$ in the case of equality. We call spanning sets attaining this bound \emph{length-maximal}. For $k=3$, these sets are higher-dimensional analogues of maximal arcs; when $h=1$ they are precisely the maximal arcs of $\mathrm{PG}(2,q)$, which for $t<q$ exist if and only if $q$ is even and $t$ divides $q$. For $k=4$ and $q^h>2$, we show that any length-maximal set must satisfy $t = q^h+1$ and that every hyperplane is either a $t$-secant or a $1$-secant. Such sets exist for all $q$ and $h$, arising as field reductions of ovoids of $\mathrm{PG}(3,q^h)$. For $k \geq 5$ and $q^h>3$, no length-maximal set exists. The case $q^h=3$, that is $(q,h)=(3,1)$, is exceptional: there, examples arising from the ternary Golay code exist for $k=5$ and $k=6$, while none exist for $k\ge7$. In the language of additive codes, these results assert that additive codes over $\mathbb{F}_{q^h}$ attaining the natural Griesmer-type bound do not exist when the code dimension is $5$ or more and $q^h>3$, apart from these two sporadic $\mathbb{F}_3$ examples. We also determine the strongly regular graphs associated with length-maximal sets.
	\end{abstract}
	
	
	\section{Introduction}
	
	A {\em maximal arc} $\mathcal{X}$ in $\mathrm{PG}(2,q)$ is a set of points with the property that every line is incident with $0$ or $t$ points of $\mathcal{X}$. It is straightforward to prove that $|\mathcal{X}|=(t-1)q+t$ and that $t$ divides $q$. Denniston \cite{den} provided examples for $q$ even and all possible values of $t$. Further examples were discovered by Thas \cite{tha1,tha3}, Mathon \cite{mathoncon}, and Hamilton and Mathon \cite{HM2003}. When $q$ is odd, maximal arcs were shown not to exist by Ball, Blokhuis and Mazzocca \cite{BBM}; see also \cite{BB}. The parameter set $(t,q)=(3,9)$ was ruled out earlier by Cossu \cite{cos} and the parameter set $(t,q)=(3,3^{h})$ by Thas \cite{tha}.
	
	In this article we consider a natural generalisation of maximal arcs to higher dimensions. A maximal arc in $\mathrm{PG}(2,q^h)$ yields, by field reduction, a set of $(h-1)$-dimensional subspaces in $\mathrm{PG}(3h-1,q)$ with the property that every hyperplane contains either $0$ or $t$ subspaces. This leads us to study sets $\mathcal{X}$ of $(h-1)$-dimensional subspaces in $\mathrm{PG}(kh-1,q)$ with the property that every hyperplane contains at most $t$ subspaces of $\mathcal{X}$.
	
	There is an equivalent formulation in terms of additive codes. An additive $[n,k,d]_{q}^h$ code $C$ is an $\mathbb{F}_q$-linear subspace of $\mathbb{F}_{q^h}^n$ of size $q^{kh}$ in which every non-zero vector has at least $d$ non-zero coordinates. Let $G$ be a $kh \times n$ matrix whose row span is $C$. By fixing a basis $\{e_1,\ldots,e_h\}$ for $\mathbb{F}_{q^h}$ over $\mathbb{F}_q$ we can write
	$$
	G=\sum_{j=1}^h e_jG_j,
	$$
	where $G_j$ is a $kh \times n$ matrix with entries from $\mathbb{F}_q$. For each coordinate $i \in \{1,\ldots,n\}$ we define a subspace $\pi_i$ as the span of the $i$-th column of $G_j$, $j \in \{1,\ldots,h\}$. Thus $\pi_i$ is a subspace of $\mathrm{PG}(kh-1,q)$ of projective dimension at most $h-1$. The property that every non-zero vector of $C$ has at least $d$ non-zero coordinates is equivalent to the property that every hyperplane of $\mathrm{PG}(kh-1,q)$ contains at most $t = n-d$ elements of the multi-set $\mathcal{X} = \{\pi_1,\ldots,\pi_n\}$.
	
	In this setting the length bound $|\mathcal{X}| \leq (t-k+2)q^h+t$ is a Griesmer-type bound for additive codes. Codes attaining this bound are additive few-weight codes; for $k=3$ they are two-weight codes whose non-zero weights are $n$ and $n-t$. An additive permutation on each coordinate produces an equivalent code containing the all-one vector.
	
	Although maximal arcs do not exist in planes of odd characteristic, Mathon \cite{DDHM2002} discovered an example of $21$ lines in $\mathrm{PG}(5,3)$ (equivalently, $k=3$, $h=2$, $q=3$, $t=3$) with the property that every hyperplane contains $0$ or $3$ lines. This example shows that a direct generalisation of the Ball--Blokhuis--Mazzocca nonexistence theorem to higher dimensions is not possible. To our knowledge it is the only known example in odd characteristic, and the uniqueness question remains open.
	
	The main results of this paper are as follows. Theorem~\ref{mainthm} establishes the bound $|\mathcal{X}| \leq (t-k+2)q^h+t$ and characterises equality. The divisibility condition $t-k+3 \mid q^h$ (Theorem~\ref{cor: divisibility}) provides a further necessary condition for length-maximal sets. Theorem~\ref{lem: k=4 maximal arc DNE} shows that for $k=4$ and $q^h>2$, any length-maximal set must satisfy $t=q^h+1$ and have no zero-secant hyperplanes. Such sets exist for all $q$ and $h$, as field reductions of ovoids of $\mathrm{PG}(3,q^h)$ (Example~\ref{ex: ovoid}). Theorem~\ref{lem: k=5 maximal arc DNE} shows that no length-maximal set exists for $k=5$ when $q^h>3$, and by induction this extends to all $k \ge 5$. The borderline case $(q,h)=(3,1)$, is genuinely exceptional in that length-maximal sets exist for $k=5$ and $k=6$ but not for $k\ge7$ (Example~\ref{ex: golay} and Proposition~\ref{prop: q3 k7}). In Section~\ref{sec: srg} we determine the strongly regular graphs associated with length-maximal sets. The case $q=2$ is discussed in Section~\ref{sec: remarks}.
	
	Throughout, $k \ge 2$, and dimension refers to projective dimension.
	
	\section{An upper bound on the size of $\mathcal{X}$}
	
	\begin{theorem} \label{mainthm}
		Let $\mathcal{X}$ be a multi-set of $(h-1)$-dimensional subspaces in $\mathrm{PG}(kh-1,q)$ with the property that every hyperplane contains at most $t$ subspaces of $\mathcal{X}$. Then
		\begin{equation} \label{eqn: maximal arc bound}
			|\mathcal{X}| \le (t-k+2)q^h+t,
		\end{equation}
		with equality if and only if
		\begin{enumerate}
			\item every set of $(k-1)$ elements of $\mathcal{X}$ spans a $((k-1)h-1)$-dimensional subspace, and \label{part: mainthm 1}
			\item every hyperplane contains either $0, 1, \ldots, k-3$ or $t$ elements of $\mathcal{X}$. \label{part: mainthm 2}
		\end{enumerate}
	\end{theorem}
	
	\begin{proof}
		Write $n=|\mathcal{X}|$. Let $x_1,\ldots,x_{k-2}$ be elements of $\mathcal{X}$ and let $\pi=\langle x_1,\ldots,x_{k-2} \rangle$. The dimension of $\pi$ is $(k-2)h-1-e_1$ for some $e_1 \ge 0$. Let $(k-1)h-1-e_1-e_2$ denote the maximum dimension of $\langle \pi, x \rangle$ over $x \in \mathcal{X} \setminus \{x_1,\ldots,x_{k-2}\}$.
		
		We count pairs $(x,H)$ where $x \in \mathcal{X} \setminus \{x_1,\ldots,x_{k-2}\}$ and $H$ is a hyperplane containing both $x$ and $\pi$. By hypothesis each such hyperplane contains at most $t-k+2$ elements of $\mathcal{X} \setminus \{x_1,\ldots,x_{k-2}\}$, so this count is at most
		$$
		\frac{q^{2h+e_1}-1}{q-1}(t-k+2).
		$$
		Since $(k-1)h-1-e_1-e_2$ is the maximum dimension of $\langle \pi, x \rangle$, each element of $\mathcal{X} \setminus \{x_1,\ldots,x_{k-2}\}$ lies in at least
		$$
		\frac{q^{h+e_1+e_2}-1}{q-1}
		$$
		hyperplanes through $\pi$, so the count is at least
		$$
		\frac{q^{h+e_1+e_2}-1}{q-1}(n-k+2).
		$$
		Combining these inequalities,
		$$
		n \le \frac{q^{2h+e_1}-1}{q^{h+e_1+e_2}-1}(t-k+2)+k-2.
		$$
		This bound is maximised when $e_1=e_2=0$, giving $n \le (t-k+2)q^h+t$.
		
		Equality requires $e_1=e_2=0$, which means every $(k-1)$ elements of $\mathcal{X}$ span a $((k-1)h-1)$-flat (part~\ref{part: mainthm 1} of Theorem~\ref{mainthm}), and every hyperplane containing $k-2$ elements of $\mathcal{X}$ contains exactly $t$ elements of $\mathcal{X}$ (part~\ref{part: mainthm 2} of Theorem~\ref{mainthm}).
	\end{proof}
	
	The bound \eqref{eqn: maximal arc bound} is called the \emph{maximal arc bound}, and a set $\mathcal{X}$ attaining it which spans $\mathrm{PG}(kh-1,q)$ is called \emph{length-maximal}. The spanning condition excludes the degenerate case $|\mathcal{X}|=t=k-2$, in which both conditions of Theorem~\ref{mainthm} are vacuous, and corresponds to the requirement that the associated additive code have size $q^{kh}$. A length-maximal $\mathcal{X}$ corresponds to an additive few-weight code over ${\mathbb F}_{q^h}$; by part~\ref{part: mainthm 2} of Theorem~\ref{mainthm} its non-zero weights lie in $\{\,n-t,\ n,\ n-1,\ \ldots,\ n-(k-3)\,\}$, and for $k=3$ it is a two-weight code with weights $n$ and $n-t$.
	
	\begin{lemma}\label{lem: no maximal MDS codes}
		If $\mathcal{X}$ is a length-maximal set of $(h-1)$-dimensional subspaces in $\mathrm{PG}(kh-1, q)$ with $q^h>2$ and $k\ge 4$, then $t\ge k$.
	\end{lemma}
	
	\begin{proof}
		Since $\mathcal{X}$ spans $\mathrm{PG}(kh-1,q)$, we have $|\mathcal{X}| \ge k$. If $t \le k-2$ then $t-k+2 \le 0$, so $|\mathcal{X}| = (t-k+2)q^h+t \le t \le k-2$, a contradiction; hence $t \ge k-1$. If $t=k-1$ then $|\mathcal{X}|=q^h+k-1$ and the corresponding additive code is a $[q^h+k-1,k,q^h]$ MDS code over $\mathbb{F}_{q^h}$. Since $k \ge 4$ and $q^h>2$, a classical result of Bruen and Silverman \cite[Theorem~6]{MR704102} requires $36 \mid q^h$, whence $t \ge k$.
	\end{proof}
	
	\begin{lemma} \label{lem: single weight bound k=3}
		Let $\mathcal{X}$ be a set of $(h-1)$-dimensional subspaces in $\mathrm{PG}(3h-1, q)$ with the property that every hyperplane contains at most $t$ elements of $\mathcal{X}$. If $|\mathcal{X}| = (t - 1)q^h + t$ then $t\leq q^h+1$. If $t < q^h + 1$ then there exists a hyperplane disjoint from $\mathcal{X}$. If $t = q^h + 1$ then every hyperplane contains exactly $t$ elements of $\mathcal{X}$.
	\end{lemma}
	
	\begin{proof}
		Counting pairs $(x, H)$ where $x \in \mathcal{X}$ and $H$ is a hyperplane containing $x$,
		$$
		|\mathcal{X}| \cdot \frac{q^{2h}-1}{q-1} \le t \cdot \frac{q^{3h}-1}{q-1},
		$$
		with equality if and only if every hyperplane contains exactly $t$ elements of $\mathcal{X}$. Substituting $|\mathcal{X}|=(t-1)q^h+t$ and simplifying,
		$$
		(t-1)q^h+t \leq t\,\frac{q^{3h}-1}{q^{2h}-1} = t\left(q^h + \frac{1}{q^h+1}\right).
		$$
		This rearranges to $q^h(q^h-t+1) \ge 0$, which holds if and only if $t \le q^h+1$, with equality precisely when $t=q^h+1$ and every hyperplane is a $t$-secant. For $t<q^h-1$, apply Theorem~\ref{mainthm},  part~\ref{part: mainthm 2}.
	\end{proof}
	
	We next establish a divisibility condition satisfied by $t$ in the extremal case. 
		
	\begin{lemma} \label{dividinglemma}
		Let $\mathcal{X}$ be a multi-set of $(h-1)$-dimensional subspaces in $\mathrm{PG}(kh-1,q)$ of size $(t-k+2)q^h+t$ with the property that every hyperplane contains either $0$ or $t$ subspaces of $\mathcal{X}$. Let $\pi$ be an $(s-1)$-dimensional subspace skew to $\mathcal{X}$. Then $\pi$ is contained in exactly
		$$
		\frac{q^h-q^{kh-s-h}}{q-1}+\frac{(k-2)(q^{kh-s}-q^h)}{t(q-1)}
		$$
		zero-secant hyperplanes. In particular, if there exists a point skew to $\mathcal{X}$, then $t$ divides $(k-2)q^h$.
	\end{lemma}
	
	\begin{proof}
		Let $a_0$ and $a_t$ denote the number of zero-secant and $t$-secant hyperplanes containing $\pi$, respectively. Counting hyperplanes through $\pi$,
		$$
		a_0+a_t=\frac{q^{kh-s}-1}{q-1},
		$$
		and counting pairs $(x,H)$ with $x \in \mathcal{X}$ and $H$ a hyperplane through both $\pi$ and $x$,
		$$
		ta_t=\bigl((t-k+2)q^h+t\bigr)\frac{q^{kh-s-h}-1}{q-1}.
		$$
		Solving for $a_0$ gives the stated formula. Applying this for $s=0$ and $s=1$ yields
		$$
		t \mid (k-2)q^h(1+q+\cdots+q^{kh-h-1}) \quad \text{and} \quad t \mid (k-2)q^h(1+q+\cdots+q^{kh-h-2}),
		$$
		from which $t$ divides $(k-2)q^h$.
	\end{proof}

	\begin{theorem} \label{cor: divisibility}
		Let $\mathcal{X}$ be a length-maximal set of $(h-1)$-dimensional subspaces in $\mathrm{PG}(kh-1, q)$ with $k\ge 3$ and $t < q^h+k-2$. Then $t-k+3$ divides $q^h$.
	\end{theorem}
	
	\begin{proof}
		For $k=3$, part~\ref{part: mainthm 2} of Theorem~\ref{mainthm} gives that every hyperplane contains $0$ or $t$ elements of $\mathcal{X}$.   Since $t<q^h+1$, we have $|\mathcal{X}|=(t-1)q^h+t \le q^{2h}$. The elements of $\mathcal{X}$ are pairwise disjoint by part~\ref{part: mainthm 1} of Theorem~\ref{mainthm}, so they cover at most $
		q^{2h}\,\frac{q^h-1}{q-1} < \frac{q^{3h}-1}{q-1}
		$
		points of $\mathrm{PG}(3h-1,q)$. Hence there is a point skew to $\mathcal{X}$, and Lemma~\ref{dividinglemma} gives that $t$ divides $q^h$.
		
		For $k \ge 4$, we argue inductively. By part~\ref{part: mainthm 1} of Theorem~\ref{mainthm}, taking the quotient at any element $x \in \mathcal{X}$ yields a set $\mathcal{X}'$ of $(h-1)$-dimensional subspaces in $\mathrm{PG}((k-1)h-1,q)$ of size $ ((t-1)-(k-1)+2)q^h+(t-1)$, with every hyperplane containing at most $t-1$ elements of $\mathcal{X}'$. Thus $\mathcal{X}'$ is length-maximal with parameter $t-1$ in $\mathrm{PG}((k-1)h-1,q)$. Since $t-1 < q^h+k-3$, the induction hypothesis gives $(t-1)-(k-1)+3 = t-k+3$ divides $q^h$.
	\end{proof}
	
	\section{Length-maximal sets for $k \ge 4$}\label{sec: k ge 4}
	
	We enumerate secant hyperplanes of a length-maximal set $\mathcal{X}$. A hyperplane containing precisely $s$ elements of $\mathcal{X}$ is called an \emph{$s$-secant}. For $\mathcal{X}$ length-maximal we write $N(t,h,k,s)$ for the number of $s$-secants of $\mathcal{X}$ in $\mathrm{PG}(kh-1,q)$.
	
	\begin{lemma}\label{lem: enumerating secants}
		Let $\mathcal{X}$ be a length-maximal set of $(h-1)$-dimensional subspaces in $\mathrm{PG}(kh-1,q)$. Set $n = |\mathcal{X}| = (t-k+2)q^h+t$. Then
		
		\begin{enumerate}
			\item $\displaystyle N(t,h,3,t) = \frac{q^{3h}-1}{q-1} -\frac{q^h-1}{q-1}\cdot\left(\frac{q^h(q^h+1)}{t}-q^h\right)$ \label{part 1: lem: enumerating secants}
			
			\item $\displaystyle N(t,h,3,0) = \frac{q^h-1}{q-1}\cdot\left(\frac{q^h(q^h+1)}{t}-q^h\right)$ \label{part 2: lem: enumerating secants}
			
			\item $\displaystyle N(t,h,4,t) = \left( q^h + 1 - \frac{2q^h}{t} \right) \left( q^h + 1 - \frac{q^h}{t-1} \right) \frac{q^{2h}-1}{q-1}$ \label{part 3: lem: enumerating secants}
			
			\item $\displaystyle N(t,h,4,1) = n\cdot N(t-1,h,3,0) =\left((t-2)q^h + t\right) \cdot \frac{q^h-1}{q-1}\cdot \left(  \frac{q^h(q^h+1)}{t-1}-q^h \right)$ \label{part 4: lem: enumerating secants}
			
			\item $\displaystyle N(q^h+2,h,5,2) = \frac{n}{2}\cdot N(q^h+1,h,4,1) = \frac{q^{2h}+2}{2}\cdot (q^{2h}+1)\cdot\frac{q^h-1}{q-1}$ \label{part 5: lem: enumerating secants}
			
			\item $\displaystyle N(q^h+2,h,5,q^h+2) = \frac{(q^{2h}+2)(q^{2h}+1)q^h(q^h-1)}{(q^h+2)(q-1)}$ \label{part 6: lem: enumerating secants}
		\end{enumerate}
	\end{lemma}
	
	\begin{proof}
		Set $n = (t-k+2)q^h+t$ and we use the term $(h-1)$-flat for an $(h-1)$-dimensional subspace.
		
		\medskip\noindent\textit{Part~\ref{part 1: lem: enumerating secants}.}
		By part~\ref{part: mainthm 2} of Theorem~\ref{mainthm}, every hyperplane is a $0$-secant or a $t$-secant when $k=3$. Counting pairs $(x,H)$ with $x \in \mathcal{X}$ and $H$ a $t$-secant through $x$, and using the fact that each $(h-1)$-flat in $\mathrm{PG}(3h-1,q)$ lies in $\frac{q^{2h}-1}{q-1}$ hyperplanes,
		$$
		N(t,h,3,t)\cdot t = n\cdot \frac{q^{2h}-1}{q-1}.
		$$
		Substituting $n=(t-1)q^h+t$ and rearranging gives the formula. Part~\ref{part 2: lem: enumerating secants} follows since $N(t,h,3,0)+N(t,h,3,t)=|\mathrm{PG}(3h-1,q)|$.
		
		\medskip\noindent\textit{Part~\ref{part 3: lem: enumerating secants}.}
		By part~\ref{part: mainthm 1} of Theorem~\ref{mainthm}, any two elements of $\mathcal{X}$ span a $(2h-1)$-flat; each such flat lies in $\frac{q^{2h}-1}{q-1}$ hyperplanes of $\mathrm{PG}(4h-1,q)$. By part~\ref{part: mainthm 2} of Theorem~\ref{mainthm}, each hyperplane is a $0$-, $1$-, or $t$-secant when $k=4$. Counting triples $(x,y,H)$ with $x \ne y$ in $\mathcal{X}$ and $H$ a $t$-secant through both,
		$$
		N(t,h,4,t)\cdot \binom{t}{2} = \binom{n}{2} \cdot \frac{q^{2h}-1}{q-1}.
		$$
		Substituting $n=(t-2)q^h+t$ and simplifying gives the formula.
		
		\medskip\noindent\textit{Part~\ref{part 4: lem: enumerating secants}.}
		For $k=4$, a $1$-secant hyperplane contains exactly one element $x_i \in \mathcal{X}$. The quotient of $\mathcal{X}$ at any $x_i \in \mathcal{X}$ is a length-maximal set $\mathcal{X}_i^*$ in $\mathrm{PG}(3h-1,q)$ with parameter $t-1$. Each $0$-secant of $\mathcal{X}_i^*$ corresponds to a $1$-secant of $\mathcal{X}$ through $x_i$, so summing over all $x_i$,
		$$
		N(t,h,4,1) = n \cdot N(t-1,h,3,0).
		$$
		Substituting the formula from part~\ref{part 2: lem: enumerating secants} of Lemma~\ref{lem: enumerating secants} with $t$ replaced by $t-1$ gives the result.
		
		\medskip\noindent\textit{Part~\ref{part 5: lem: enumerating secants}.}
		Now let $k=5$ and $t=q^h+2$, so $n=q^{2h}+2$. By part~\ref{part: mainthm 2} of Theorem~\ref{mainthm}, every hyperplane is a $0$-, $1$-, $2$-, or $t$-secant. Since any hyperplane containing $3$ or more elements of $\mathcal{X}$ must contain $t$ elements (as any $3$ elements span a $(3h-1)$-flat, and any hyperplane through a $(3h-1)$-flat that contains a fourth element of $\mathcal{X}$ must contain all $t$ by part~\ref{part: mainthm 2} of Theorem~\ref{mainthm}), the secant types are restricted to $0$, $1$, $2$, and $t$.
		
		For $k=5$, each element $x_i \in \mathcal{X}$ lies in a $2$-secant hyperplane $H$ if and only if $H$ contains exactly one further element $x_j \ne x_i$. The number of $2$-secant hyperplanes through $x_i$ equals the number of $1$-secant hyperplanes of the quotient set $\mathcal{X}_i^*$ in $\mathrm{PG}(4h-1,q)$, which has parameter $t-1=q^h+1$. Summing over all $x_i$ and dividing by $2$ (since each $2$-secant is counted twice),
		$$
		N(q^h+2,h,5,2) = \frac{n}{2}\cdot N(q^h+1,h,4,1).
		$$
		By Theorem~\ref{lem: k=4 maximal arc DNE} (proved below), the quotient set $\mathcal{X}_i^*$ has $t-1=q^h+1$ and has no $0$-secants, so $N(q^h+1,h,4,1) = |\mathrm{PG}(4h-1,q)| - N(q^h+1,h,4,q^h+1)$. Alternatively, substituting $t-1=q^h+1$ directly into part~\ref{part 4: lem: enumerating secants} of Lemma~\ref{lem: enumerating secants} and simplifying gives
		$$
		N(q^h+2,h,5,2) = \frac{q^{2h}+2}{2}\cdot (q^{2h}+1)\cdot\frac{q^h-1}{q-1}.
		$$
		
		\medskip\noindent\textit{Part~\ref{part 6: lem: enumerating secants}.}
		With $k=5$, $t=q^h+2$, $n=q^{2h}+2$, any three elements of $\mathcal{X}$ span a $(3h-1)$-flat (since any $k-1=4$ elements span a $(4h-1)$-flat by part~\ref{part: mainthm 1} of Theorem~\ref{mainthm}, which implies any three also span a $(3h-1)$-flat). Each $(3h-1)$-flat in $\mathrm{PG}(5h-1,q)$ lies in $\frac{q^{2h}-1}{q-1}$ hyperplanes. Counting quadruples $(x_1,x_2,x_3,H)$ with distinct $x_i \in \mathcal{X}$ and $H$ a $t$-secant through all three,
		$$
		N(q^h+2,h,5,q^h+2)\cdot \binom{t}{3} = \binom{n}{3}\cdot\frac{q^{2h}-1}{q-1}.
		$$
		Substituting $n=q^{2h}+2$ and $t=q^h+2$ and simplifying,
		$$
		N(q^h+2,h,5,q^h+2) = \frac{(q^{2h}+2)(q^{2h}+1)q^{2h}}{(q^h+2)(q^h+1)q^h}\cdot\frac{q^{2h}-1}{q-1} = \frac{(q^{2h}+2)(q^{2h}+1)q^h(q^h-1)}{(q^h+2)(q-1)}.
		$$
	\end{proof}
	
	\begin{theorem}\label{lem: k=4 maximal arc DNE}
		Let $\mathcal{X}$ be a length-maximal set of $(h-1)$-dimensional subspaces in $\mathrm{PG}(4h-1, q)$ with $q^h > 2$. Then $t=q^h+1$ and every hyperplane is either a $t$-secant or a $1$-secant of $\mathcal{X}$.
	\end{theorem}
	
	\begin{proof}
		By part~\ref{part: mainthm 2} of Theorem~\ref{mainthm}, every hyperplane is a $0$-, $1$-, or $t$-secant. The counts $N(t,h,4,0)$, $N(t,h,4,1)$, $N(t,h,4,t)$ are non-negative integers summing to $|\mathrm{PG}(4h-1,q)|$. Using parts~\ref{part 3: lem: enumerating secants} and~\ref{part 4: lem: enumerating secants} of Lemma~\ref{lem: enumerating secants},
		\begin{align}
		N(t,h,4,0) & = |\mathrm{PG}(4h-1,q)| - N(t,h,4,t) - N(t,h,4,1) \nonumber \\ 
		& = \frac{q^h(q^h-1)(t-q^h-1)\bigl(t(q^h+1)-4q^h-2\bigr)}{t(q-1)}. \label{eqn: 4_0}
		\end{align}	
		By Lemma~\ref{lem: no maximal MDS codes}, $t \ge 4$. Taking the quotient of $\mathcal{X}$ at any of its elements yields a length-maximal set in $\mathrm{PG}(3h-1,q)$ with parameter $t-1$, so Lemma~\ref{lem: single weight bound k=3} gives $t-1 \le q^h+1$, that is, $t \le q^h+2$. We rule out every value other than $t=q^h+1$.
		
		For $4 \le t \le q^h$,  $(t-q^h-1) \le -1 $, and $t(q^h+1)-(4q^h+2) \ge 4(q^h+1)-(4q^h+2) = 2 > 0$, so from (\ref{eqn: 4_0}) $N(t,h,4,0) < 0$, a contradiction.

		For $t = q^h+2$, substituting into part~\ref{part 4: lem: enumerating secants} of Lemma~\ref{lem: enumerating secants} gives $N(t,h,4,1)=0$, so every hyperplane is a $0$- or $t$-secant. Since $|\mathcal{X}| = q^{2h}+q^h+2$ and each element is an $(h-1)$-flat with $\frac{q^h-1}{q-1}$ points, the elements of $\mathcal{X}$ cover at most $
		(q^{2h}+q^h+2)\cdot\frac{q^h-1}{q-1} < \frac{q^{4h}-1}{q-1}
		$
		points of $\mathrm{PG}(4h-1,q)$, so some point is skew to $\mathcal{X}$. Lemma~\ref{dividinglemma} then gives $t \mid 2q^h$; but $t=q^h+2$ divides $2q^h = 2(q^h+2)-4$ only if $q^h+2 \mid 4$, giving $q^h \le 2$. 
		Therefore $t=q^h+1$, and substituting gives $N(t,h,4,0)=0$.
	\end{proof}
	
	\begin{theorem}\label{lem: k=5 maximal arc DNE}
		There is no length-maximal set of $(h-1)$-dimensional subspaces in $\mathrm{PG}(5h-1, q)$ when $q^h > 3$.
	\end{theorem}
	
	\begin{proof}
		Suppose for a contradiction that $\mathcal{X}$ is such a set, with $|\mathcal{X}|=(t-3)q^h+t$. The quotient of $\mathcal{X}$ at any element is a length-maximal set in $\mathrm{PG}(4h-1,q)$ with parameter $t-1$. By Theorem~\ref{lem: k=4 maximal arc DNE}, we obtain $t-1=q^h+1$, that is $t=q^h+2$. Hence $n=|\mathcal{X}|=q^{2h}+2$.
		
		The secant types are $0$, $1$, $2$, and $t=q^h+2$. Consider the quantity
		\begin{equation} \label{eqn: difference k = 5}
			\delta = N(q^h+2,h,5,q^h+2) + N(q^h+2,h,5,2) - |\mathrm{PG}(5h-1,q)|,
		\end{equation}
		which satisfies $\delta \le 0$ since $N(q^h+2,h,5,0)$ and $N(q^h+2,h,5,1)$ are non-negative. Substituting parts~\ref{part 5: lem: enumerating secants} and~\ref{part 6: lem: enumerating secants} of Lemma~\ref{lem: enumerating secants} and expanding, we obtain
		\begin{equation}\label{eqn: delta k = 5}
			2(q^h+2)(q-1)\,\delta = (q^h)^6 - 5(q^h)^5 + 7(q^h)^4 - 3(q^h)^3 = (q^h)^3(q^h-1)^2(q^h-3).
		\end{equation}
		
		 Since $q^h > 3$, this gives $\delta > 0$, contradicting $\delta \le 0$. (When $q^h=3$ the right-hand side of~\eqref{eqn: delta k = 5} vanishes and there are two secant types.  This case is genuinely exceptional, see Example~\ref{ex: golay}.)
	\end{proof}

\begin{remark}\label{rem: induction}
	By induction on the quotient construction, Theorem~\ref{lem: k=5 maximal arc DNE} extends to all $k \ge 5$ when $q^h > 3$.
\end{remark}

	\begin{example}\label{ex: golay}
		For $(q,h)=(3,1)$, length-maximal sets exist both for $k=5$ and for $k=6$. In each case $t=k$ and $|\mathcal{X}|=k+6$, so $\mathcal{X}$ is a set of $k+6$ points of $\mathrm{PG}(k-1,3)$ giving a $[k+6,k,6]_3$ code. In each case the dual code has minimum distance $k$, so every $k-1$ of the points are linearly independent and part~\ref{part: mainthm 1} of Theorem~\ref{mainthm} holds.

		For $k=5$, let $\mathcal{X}$ be the $11$ points of $\mathrm{PG}(4,3)$ given by the columns of a generator matrix of the $[11,5,6]_3$ code dual to the perfect ternary Golay code. This is a projective two-weight code with weights $6$ and $9$, so every hyperplane meets $\mathcal{X}$ in $11-6=5$ or $11-9=2$ points; thus $t=5=q^h+2$ and the secant types are $2$ and $t$. Here the right-hand side of~\eqref{eqn: delta k = 5} vanishes.

		For $k=6$, let $\mathcal{X}$ be the $12$ points of $\mathrm{PG}(5,3)$ given by the columns of a generator matrix of the self-dual extended ternary Golay code $[12,6,6]_3$. Its weights are $6$, $9$ and $12$, so every hyperplane meets $\mathcal{X}$ in $6$, $3$ or $0$ points; thus $t=6$ and the secant types are $0$, $3=k-3$ and $t$. In particular this length-maximal set is a three-weight code.
	\end{example}

	
	\begin{prop}\label{prop: q3 k7}
		For $q^h>2$ there is no length-maximal set with $k\ge 7$.
	\end{prop}

	\begin{proof} From Theorem \ref{lem: k=5 maximal arc DNE} it suffices to consider $(q,h)=(3,1)$. 
		Iterating the quotient construction down to dimension $5$ and applying Theorem~\ref{lem: k=4 maximal arc DNE}, any length-maximal set with $(q,h)=(3,1)$ and $k\ge5$ has $t=k$ and $|\mathcal{X}|=(t-k+2)\cdot3+t=k+6$. Suppose such a set $\mathcal{X}$ exists with $k=7$, so that $t=7$ and $n=|\mathcal{X}|=13$ in $\mathrm{PG}(6,3)$. By part~\ref{part: mainthm 1} of Theorem~\ref{mainthm}, any six points of $\mathcal{X}$ span a unique hyperplane. By part~\ref{part: mainthm 2} of Theorem~\ref{mainthm} such a hyperplane must be a $t$-secant. Counting the pairs $(S,H)$ in which $S$ is a six-element subset of $\mathcal{X}$ and $H$ is the hyperplane spanned by $S$ gives
		$$
		\binom{t}{6}\,N(7,1,7,7) = \binom{n}{6}, \qquad\text{so}\qquad N(7,1,7,7)=\frac{\binom{13}{6}}{\binom{7}{6}}=\frac{1716}{7},
		$$
		which is not an integer. This rules out $k=7$, and the cases $k\ge8$ are ruled out inductively.
	\end{proof}

		\begin{remark}\label{rem: AZ conjectures}
		Restricting to the case $h=1$, let $\kappa(s,q)$ denote, as in~\cite{AZ}, the largest $k$ for which a length-maximal linear code of Singleton defect $s$ over $\mathbb{F}_q$ exists.
		Conjecture~5.1 of~\cite{AZ} asserts that $\kappa(s,q)\le4$ for $q>3$, so Theorem~\ref{lem: k=5 maximal arc DNE}, together with the quotient induction in Remark~\ref{rem: induction}, establishes it. Furthermore, Theorem~\ref{lem: k=4 maximal arc DNE} shows that a length-maximal set with $k=4$ has $t=q+1$, that is $s=q-2$, while Example~\ref{ex: ovoid} shows that such sets exist for every $q$. Hence
		$
		\kappa(q-2,q)=4 $ $ (\text{for }q>3),
		$
		which is the third of the three cases of Conjecture~5.2 of~\cite{AZ}.
	\end{remark}
	
	\section{The associated strongly regular graphs}\label{sec: srg}

	The classical correspondence among projective two-weight codes, two-intersection sets and strongly regular graphs is well known; see Delsarte~\cite{Del} and Calderbank and Kantor~\cite{CK}. In this section we interpret the preceding results from the graph-theoretic side.

	Let $\mathcal{X}$ be a length-maximal set of $(h-1)$-dimensional subspaces of $\mathrm{PG}(kh-1,q)$ with $k\ge3$ and $n=|\mathcal{X}|$. By part~\ref{part: mainthm 1} of Theorem~\ref{mainthm} the elements of $\mathcal{X}$ are pairwise disjoint, so the $h$-dimensional $\mathbb{F}_q$-subspaces $V_1,\ldots,V_n$ of $\mathbb{F}_q^{kh}$ that they determine meet pairwise in the zero vector. Put
	$$
	D=\Bigl(\bigcup_{i=1}^{n} V_i\Bigr)\setminus\{0\},\qquad |D|=n(q^h-1),
	$$
	and let $\Gamma(\mathcal{X})$ denote the Cayley graph on $(\mathbb{F}_q^{kh},+)$ with connection set $D$. Thus $\Gamma(\mathcal{X})$ has $q^{kh}$ vertices and is regular of degree $n(q^h-1)$.

	Let $S$ denote the set of points of $\mathrm{PG}(kh-1,q)$ covered by the members of $\mathcal{X}$. Since these members are pairwise disjoint,
	$$
	|S|=N:=n\,\frac{q^h-1}{q-1},
	$$
	and $S$ spans $\mathrm{PG}(kh-1,q)$ because $\mathcal{X}$ does. Thus $S$ is a spanning set of $N$ points of $\mathrm{PG}(kh-1,q)$, giving rise to a projective $[N,kh]_q$ code in the sense of Calderbank and Kantor~\cite[Section~2]{CK}, and $D=\{v\in\mathbb{F}_q^{kh}\setminus\{0\}:\langle v\rangle\in S\}$. Consequently $\Gamma(\mathcal{X})$ is the graph that~\cite[Section~3]{CK} associates with $S$, and the results of that section apply verbatim.

	\begin{prop}\label{prop: srg}
		Let $\mathcal{X}$ be a length-maximal set of $(h-1)$-dimensional subspaces of $\mathrm{PG}(kh-1,q)$ with $k\ge3$, and let $n=|\mathcal{X}|$. If $H$ is an $s$-secant hyperplane of $\mathcal{X}$ then the eigenvalue of $\Gamma(\mathcal{X})$ afforded by the corresponding character is
		$$
		\theta_H=sq^h-n.
		$$
		Consequently $\Gamma(\mathcal{X})$ is strongly regular if and only if $\mathcal{X}$ has exactly two secant types $s_1<s_2$, in which case its parameters are
		$$
		v=q^{kh},\qquad K=n(q^h-1),\qquad \lambda=K+r+\sigma+r\sigma,\qquad \mu=K+r\sigma,
		$$
		where $r=s_2q^h-n$ and $\sigma=s_1q^h-n$ are the restricted eigenvalues.
	\end{prop}

	\begin{proof}
		Let $H$ be an $s$-secant hyperplane of $\mathcal{X}$. Each of the $n-s$ members of $\mathcal{X}$ not contained in $H$ meets $H$ in a hyperplane of itself, so
		$$
		|S\cap H|=s\,\frac{q^h-1}{q-1}+(n-s)\,\frac{q^{h-1}-1}{q-1},
		$$
		and the codeword of the projective code determined by $H$ has weight
		$$
		w_H=|S|-|S\cap H|=(n-s)q^{h-1}.
		$$
		By~\cite[Lemma~3.4]{CK} the eigenvalue of $\Gamma(\mathcal{X})$ afforded by the corresponding character is $K-qw_H$, where $K=N(q-1)=n(q^h-1)$ is the valency; hence
		$$
		\theta_H=n(q^h-1)-q(n-s)q^{h-1}=sq^h-n,
		$$
		as claimed. The map $s\mapsto sq^h-n$ is injective, so the restricted eigenvalues of $\Gamma(\mathcal{X})$ correspond bijectively to the secant types of $\mathcal{X}$.

		If $S$ is the whole of $\mathrm{PG}(kh-1,q)$ then $D=\mathbb{F}_q^{kh}\setminus\{0\}$, the graph $\Gamma(\mathcal{X})$ is complete and there is a single secant type. Otherwise $S$ is a proper, non-empty, spanning set of points, and by~\cite[Theorem~3.2]{CK} the graph $\Gamma(\mathcal{X})$ is strongly regular if and only if $S$ is a two-intersection set, that is, if and only if $\mathcal{X}$ has exactly two secant types $s_1<s_2$. In that case the associated code is a projective two-weight code with weights $w_1=(n-s_2)q^{h-1}$ and $w_2=(n-s_1)q^{h-1}$ by~\cite[Theorem~3.1]{CK}, and the parameters of $\Gamma(\mathcal{X})$ given in~\cite[Corollary~3.7]{CK} are, in terms of the restricted eigenvalues $r=K-qw_1=s_2q^h-n$ and $\sigma=K-qw_2=s_1q^h-n$, precisely $\mu=K+r\sigma$ and $\lambda=K+r+\sigma+r\sigma$.
	\end{proof}

	\begin{cor}\label{cor: srg list}
		Let $\mathcal{X}$ be a length-maximal set of $(h-1)$-dimensional subspaces of $\mathrm{PG}(kh-1,q)$.
		\begin{enumerate}
			\item If $k=3$ and $t<q^h+1$ then $\Gamma(\mathcal{X})$ is strongly regular with parameters
			$$
			\bigl(q^{3h},\ n(q^h-1),\ n(t-2)+q^h-t,\ n(t-1)\bigr),\qquad n=(t-1)q^h+t,
			$$
			the restricted eigenvalues being $r=q^h-t$ and $\sigma=-n$. \label{part: srg k=3}

			\item If $k=4$ and $q^h>2$ then $\Gamma(\mathcal{X})$ is strongly regular with parameters
			$$
			\bigl(q^{4h},\ (q^{2h}+1)(q^h-1),\ q^h-2,\ q^h(q^h-1)\bigr),
			$$
			the restricted eigenvalues being $r=q^h-1$ and $\sigma=q^h-q^{2h}-1$. \label{part: srg k=4}

			\item If $k\ge5$ and $q^h>3$ then no such $\mathcal{X}$ exists. For $(q,h)=(3,1)$, length-maximal sets exist precisely for $k\in\{5,6\}$: for $k=5$ the graph $\Gamma(\mathcal{X})$ is strongly regular with parameters $(243,22,1,2)$; for $k=6$ the secant types are $0$, $3$ and $6$, so $\Gamma(\mathcal{X})$ is not strongly regular. \label{part: srg k>=5}
		\end{enumerate}
	\end{cor}

	\begin{proof}
		In each case we determine the secant types and appeal to Proposition~\ref{prop: srg}. For $k=3$, part~\ref{part: mainthm 2} of Theorem~\ref{mainthm} gives the types $0$ and $t$, and by Lemma~\ref{lem: single weight bound k=3} both occur when $t<q^h+1$; substituting $s_1=0$ and $s_2=t$ gives $r=tq^h-n=q^h-t$ and $\sigma=-n$, whence $\mu=K+r\sigma=n(t-1)$ and $\lambda=\mu+r+\sigma=n(t-2)+q^h-t$. For $k=4$ and $q^h>2$, Theorem~\ref{lem: k=4 maximal arc DNE} gives $t=q^h+1$, $n=q^{2h}+1$ and the two secant types $1$ and $t$; substituting yields $r=q^h-1$, $\sigma=q^h-q^{2h}-1$, $\mu=q^h(q^h-1)$ and $\lambda=q^h-2$. For $k\ge5$ the existence statement is Theorem~\ref{lem: k=5 maximal arc DNE} together with Proposition~\ref{prop: q3 k7}, and the secant types for $k=5$ and $k=6$ are read off from Example~\ref{ex: golay}: for $k=5$ they are $2$ and $5$, so that $n=11$, $K=22$, $r=4$ and $\sigma=-5$, giving $(243,22,1,2)$; for $k=6$ there are three types.
	\end{proof}

\begin{remark}\label{rem: srg identification}
		For $h=1$ the graphs of part~\ref{part: srg k=3} of Corollary~\ref{cor: srg list} are the strongly regular graphs classically associated with maximal arcs of $\mathrm{PG}(2,q)$; a hyperoval of $\mathrm{PG}(2,4)$, for instance, gives $\mathrm{SRG}(64,18,2,6)$. For $h=1$ the sets of part~\ref{part: srg k=4} are the ovoids of $\mathrm{PG}(3,q)$, and for an elliptic quadric the graph is the affine polar graph $\mathrm{VO}^-(4,q)$; inequivalent ovoids may give non-isomorphic graphs with the same parameters, cf.~\cite{CK}. For $h>1$, the examples in part~\ref{part: srg k=4} furnished by Example~\ref{ex: ovoid} arise by restriction of scalars from ovoids of $\mathrm{PG}(3,q^h)$. The graph with parameters $(243,22,1,2)$ in part~\ref{part: srg k>=5} is the Berlekamp--van Lint--Seidel graph~\cite{BvLS}, the coset graph of the perfect ternary Golay code.
	\end{remark}

	\begin{remark}
		When $k=3$ and $t=q^h+1$ every hyperplane is a $t$-secant by Lemma~\ref{lem: single weight bound k=3}, so $D$ is all of $\mathbb{F}_q^{3h}\setminus\{0\}$ and $\Gamma(\mathcal{X})$ is complete; this is the reason for the hypothesis $t<q^h+1$ in part~\ref{part: srg k=3} of Corollary~\ref{cor: srg list}. With this boundary case set aside, Corollary~\ref{cor: srg list} gives the complete list: the associated graph is non-trivially strongly regular when $k=3$ and $t<q^h+1$, when $k=4$ and $q^h>2$, and in the single exceptional case $k=5$, $(q,h)=(3,1)$. In the case $k=4$ it is Theorem~\ref{lem: k=4 maximal arc DNE} --- specifically, the exclusion of $0$-secants --- that reduces the possible secant types from three to two.
	\end{remark}

	\section{Remarks on small cases and open questions}\label{sec: remarks}
	
	For $k=2$, a length-maximal set is a set of $(h-1)$-dimensional subspaces of $\mathrm{PG}(2h-1,q)$ such that every hyperplane contains at most $t$ elements, with $|\mathcal{X}|=tq^h+t = t(q^h+1)$. This corresponds to a partition of $\mathrm{PG}(2h-1,q)$ into $(h-1)$-flats (a spread), which exists for all $q$ and $h$.
	
	For $k=3$, length-maximal sets correspond to the higher-dimensional analogues of maximal arcs. The field reduction of a maximal arc in $\mathrm{PG}(2,q^h)$ yields an example for every parameter $t \mid q^h$ and all $q$ even. Mathon's example \cite{DDHM2002} provides an instance with $k=3$, $h=2$, $q=3$, $t=3$ in odd characteristic, and to our knowledge is the only such example in odd characteristic. Whether this example is unique, and whether further examples exist in odd characteristic for larger parameters, are questions that remain open.

	\begin{example}\label{ex: ovoid}
		For $k=4$ the bound of Theorem~\ref{lem: k=4 maximal arc DNE} is attained for every prime power $q$ and every $h\ge1$, by field reduction of an ovoid. Let  $\mathcal{O}$ be an ovoid of $\mathrm{PG}(3,q^h)$: a set of $q^{2h}+1$ points meeting every plane in $1$ or $q^h+1$ points. Elliptic quadrics are ovoids and exist for every $q^h$. Regard $\mathbb{F}_q^{4h}$ as a $4h$-dimensional vector space over $\mathbb{F}_q$, and to each point $\langle v\rangle_{\mathbb{F}_{q^h}}$ of $\mathrm{PG}(3,q^h)$ associate the $(h-1)$-flat $\mathrm{PG}_{\mathbb{F}_q}(\mathbb{F}_{q^h} v)$ of $\mathrm{PG}(4h-1,q)$. Set $
		\mathcal{X} = \{\,\mathrm{PG}_{\mathbb{F}_q}(\mathbb{F}_{q^h} v) : \langle v\rangle_{\mathbb{F}_{q^h}}\in\mathcal{O}\,\}.
		$
		An $\mathbb{F}_q$-hyperplane of $\mathbb{F}_q^{4h}$ has the form $\ker\bigl(x\mapsto\mathrm{Tr}_{\mathbb{F}_{q^h}/\mathbb{F}_q}(a\cdot x)\bigr)$ for some nonzero $a\in\mathbb{F}_q^{4h}$, and by nondegeneracy of the trace form $\mathrm{PG}_{\mathbb{F}_q}(\mathbb{F}_{q^h} v)$ is contained in it if and only if $a\cdot v=0$. Thus the $\mathbb{F}_q$-hyperplanes correspond to the planes $a\cdot X=0$ of $\mathrm{PG}(3,q^h)$, an element of $\mathcal{X}$ lying in a hyperplane exactly when the corresponding point lies in the corresponding plane. Hence every hyperplane meets $\mathcal{X}$ in $1$ or $q^h+1$ members, so $t=q^h+1$ and
		$$
		|\mathcal{X}| = q^{2h}+1 = (t-2)q^h+t.
		$$
		Therefore $\mathcal{X}$ is length-maximal, with every hyperplane a $1$-secant or a $t$-secant, and the associated two-weight additive code has weights $q^{2h}$ and $q^{2h}-q^h$. When $h=1$ these length-maximal sets are precisely the ovoids of $\mathrm{PG}(3,q)$.
	\end{example}

	The results of Section~\ref{sec: k ge 4} are constrained to $q^h>2$, and Theorem~\ref{lem: k=5 maximal arc DNE} requires $q^h>3$. In particular the binary case $q=2$ is excluded only when $h=1$, so these results do apply to $\mathbb{F}_2$-linear codes over $\mathbb{F}_{2^h}$ for every $h\ge2$. For instance, taking $q=2$ and $h=2$, Theorem~\ref{lem: k=4 maximal arc DNE} shows that a length-maximal set in $\mathrm{PG}(7,2)$ has $t=5$ and consists of $17$ lines, as realised by Example~\ref{ex: ovoid} from an ovoid of $\mathrm{PG}(3,4)$, while Theorem~\ref{lem: k=5 maximal arc DNE} shows that none exists in $\mathrm{PG}(9,2)$.  For   $(q,h)=(2,1)$, the trivial binary $[k+1,k,2]$ MDS codes, the frames $\{\langle e_1\rangle,\ldots,\langle e_k\rangle,\langle e_1+\cdots+e_k\rangle\}$ of $\mathrm{PG}(k-1,2)$, with $t=k-1$, provide length-maximal sets for every $k$. The existence question for $(q,h)=(2,1)$ is completely settled by results in the literature. A bound of Faldum and Willems in~\cite{AZ}, gives $t \le q+k-2$ for every length-maximal set with $h=1$ and $k\ge3$, so for $q=2$ only $t=k-1$ and $t=k$ can occur. The MDS examples above realise $t=k-1$ for every $k$, and Alderson and Zhang~\cite{AZ} show that length-maximal sets with $t=k$ exist precisely for $k\le4$. For $k=4$ such a set is the complement of a plane in $\mathrm{PG}(3,2)$, an $8$-cap, with secant types $0$ and $4$.

	In the language of additive codes, our results assert that additive $[n,k,d]_{q}^h$ codes meeting the maximal arc bound do not exist for $k \geq 5$ when $q^h > 3$. The case $q^h=3$, namely $(q,h)=(3,1)$, is exceptional in that the ternary Golay code yields length-maximal sets for $k=5$ and $k=6$, while none exist for $k\ge7$ (Example~\ref{ex: golay} and Proposition~\ref{prop: q3 k7}). For $k=4$ and $q^h > 2$, any such code must satisfy $t=q^h+1$ and have exactly two nonzero weights $n-1$ and $n-t$; such codes exist for all $q$ and $h$, arising by field reduction of an ovoid of $\mathrm{PG}(3,q^h)$ (Example~\ref{ex: ovoid}).

	\section*{Acknowledgements}

	We thank Adri\'an Z\'ame\v{c}n\'ik, V\'aclav Rozho\v{n} and Robert \v{S}\'amal of Charles University for sending us a note generated using \texttt{bolzano.app} that drew our attention to the construction in Example~\ref{ex: ovoid}. We also thank the referee for a careful reading and for several suggestions which improved the paper, in particular for prompting the discussion of Section~\ref{sec: srg}.

	\bibliographystyle{plain}

\begin{thebibliography}{}
		
		\bibitem{AZ}
		T.L. Alderson and Z. Zhang, Projective systems and bounds on the length of codes of non-zero defect, {\it J. Algebra Combin. Discrete Struct. Appl.}, to appear. arXiv:2504.19325.
		
		\bibitem{BBM}
		S. Ball, A. Blokhuis and F. Mazzocca, Maximal arcs in Desarguesian planes of odd order do not exist, {\it Combinatorica}, {\bf 17} (1997) 31--41.
		
		\bibitem{BB}
		S. Ball and A. Blokhuis, An easier proof of the maximal arcs conjecture, {\it Proc. Amer. Math. Soc.}, {\bf 126} (1998) 3377--3380.
		
		\bibitem{BvLS}
		E.R.\ Berlekamp, J.H.\ van Lint and J.J.\ Seidel, A strongly regular graph derived from the perfect ternary Golay code, in: {\it A Survey of Combinatorial Theory}, North-Holland, Amsterdam, 1973, pp.\ 25--30.

		\bibitem{MR704102}
		A.A. Bruen and R. Silverman, On the nonexistence of certain M.D.S. codes and projective planes, {\it Math. Z.}, {\bf 183} (1983) 171--175.
		
		\bibitem{CK}
		R.\ Calderbank and W.M.\ Kantor, The geometry of two-weight codes, {\it Bull. London Math. Soc.}, {\bf 18} (1986) 97--122.

		\bibitem{cos}
		A.\ Cossu, Su alcune propriet\`{a} dei $\{k;n\}$-archi di un piano proiettivo sopra un corpo finito, {\it Rend. Mat. e Appl.}, {\bf 20} (1961) 271--277.
		
		\bibitem{DDHM2002}
		F. De Clerck, M. Delanote, N. Hamilton and R. Mathon, Perp-systems and partial geometries, {\it Adv. Geom.}, {\bf 2} (2002) 1--12.
		
		\bibitem{Del}
		P.\ Delsarte, Weights of linear codes and strongly regular normed spaces, {\it Discrete Math.}, {\bf 3} (1972) 47--64.

		\bibitem{den}
		R.H.F.\ Denniston, Some maximal arcs in finite projective planes, {\it J. Combin. Theory}, {\bf 6} (1969) 317--319.
		
		\bibitem{HM2003}
		N. Hamilton and R. Mathon, More maximal arcs in Desarguesian projective planes and their geometric structure, {\it Adv. Geom.}, {\bf 3} (2003) 251--261.
		
		\bibitem{mathoncon}
		R. Mathon, New maximal arcs in Desarguesian planes, {\it J. Combin. Theory Ser. A}, {\bf 97} (2002) 353--368.
		
		\bibitem{tha1}
		J.A.\ Thas, Construction of maximal arcs and partial geometries, {\it Geom. Dedicata}, {\bf 3} (1974) 61--64.
		
		\bibitem{tha}
		J.A.\ Thas, Some results concerning $\{(q+1)(n-1);n\}$-arcs and $\{(q+1)(n-1)+1;n\}$-arcs in finite projective planes of order $q$, {\it J. Combin. Theory Ser. A}, {\bf 19} (1975) 228--232.
		
		\bibitem{tha3}
		J.A.\ Thas, Construction of maximal arcs and dual ovals in translation planes, {\it Europ. J. Combinatorics}, {\bf 1} (1980) 189--192.
		
	\end{thebibliography}

\end{document}